\documentclass[10.9pt,a4paper]{amsart}
\usepackage{a4wide}

\usepackage[utf8]{inputenc}
\usepackage[english]{babel}

\usepackage{amsfonts,amssymb,amsmath}
\usepackage{mathtools}

\usepackage{xcolor}
\usepackage[colorlinks,
    linkcolor={red!50!black},
    citecolor={blue!50!black},
    urlcolor={blue!80!black}]{hyperref}

\usepackage{amsthm}
\usepackage{amssymb}

\usepackage{mathrsfs}
\usepackage{graphics}
\usepackage{wrapfig}
\usepackage{bbm}
\usepackage{verbatim}
\usepackage{enumitem}

\usepackage{xcolor}
\usepackage{hyperref}

\usepackage{tikz}
\usepackage[all]{xy}
\usepackage{enumitem}

\makeatletter

\newcommand{\mylabel}[2]{#2\def\@currentlabel{#2}\label{#1}}
\usetikzlibrary{calc, matrix, arrows, cd}

\newcommand{\bsm}{\left(\begin{smallmatrix}}
\newcommand{\esm}{\end{smallmatrix}\right)}
\newtheorem{innercustomthm}{Theorem}
\newenvironment{customthm}[1]
  {\renewcommand\theinnercustomthm{#1}\innercustomthm}
  {\endinnercustomthm}

\newtheorem{theorem}{Theorem}[section]

\newtheorem{corollary}[theorem]{Corollary}
\newtheorem{lemma}[theorem]{Lemma}
\newtheorem{proposition}[theorem]{Proposition}

\theoremstyle{definition}
\newtheorem{definition}[theorem]{Definition}

\newtheorem{remark}[theorem]{Remark}

\newtheorem*{claim*}{Claim}

\newcommand{\Z}{\mathbb{Z}}

\newcommand{\F}{\mathbb{F}}

\newcommand{\Hom}{\operatorname{Hom}}
\newcommand{\GL}{\operatorname{GL}}
\newcommand{\SL}{\operatorname{SL}}

\newcommand{\map}{\rightarrow}

\newcommand{\fund}{\pi_1(X)}
\newcommand{\zpi}{{\Z[\fund]}}

\newcommand{\pair}{(X,Y)}
\newcommand{\pairm}{(X,Y;M)}
\newcommand{\pairtil}{(\widetilde{X},\widetilde{Y})}

\newcommand{\xl}{X_L}

\newcommand{\xlpr}{X_{L'}}

\newcommand{\xlpair}{{(\xlpr,\xl)}}

\newcommand{\rk}{\mbox{rk}}

\newcommand{\dd}[1]{#1\dots#1}

\newcommand{\GamTen}{\widetilde{\Gamma}\otimes}

\newtheorem{case}{Case}
\newtheorem{subcase}{Case}
\numberwithin{subcase}{case}

\title{A Torres formula for twisted Reidemeister torsion}
\author{Peter Seokhee Seong}
\email{peterseong@utexas.edu}

\begin{document}

\begin{abstract}
The Torres formula, which relates the Alexander polynomial of a link to the Alexander polyomial of its sublinks, admits a generalization to the twisted setting due to Morifuji.
This paper uses twisted Reidemeister torsion to obtain a second proof of Morifuji's result that is closer in appearance to Torres' original formula.
%further generalize Morifuji's result. 
\end{abstract}

\maketitle

\section{Introduction}

The \emph{Torres formula} relates the multivariable Alexander polynomial~$\Delta_L$ of a~$\mu$-component ordered link~$L =K_1 \cup \ldots \cup K_\mu \subset S^3$ to the multivariable Alexander polynomial of the sublink~$L':=L \setminus K_\mu$:
\[
\displaystyle
\Delta_L(t_1,\ldots,t_{\mu-1},1)\doteq
\begin{dcases}
\dfrac{t_1^{\ell k(K_1,K_2)}-1}{t_1-1}\Delta_{L'}(t_1)  &\quad \text{if } \mu=2 \\
(t_1^{\ell k(K_1,K_\mu)}\cdots t_{\mu-1}^{\ell k(K_{\mu-1},K_\mu)}-1)\Delta_{L'}(t_1,\ldots,t_{\mu-1}) &\quad \text{if } \mu>2.
\end{dcases}
\]
Torres' original proof made use of Fox calculus~\cite{Torres} but there are now also proofs involving Reidemeister torsion~\cite{TuraevReidemeister} and Seifert surfaces~\cite{CimasoniStudying}.

More recently, the Alexander polynomial has been generalized to include the additional data of a representation~$\rho \colon \pi_1(S^3 \setminus L) \to \GL(n,R)$ of the link group.
The resulting \emph{twisted Alexander polynomials}, originally due to Lin~\cite{Lin} and Wada~\cite{Wada},  were later reformulated using Reidemeister torsion~\cite{Kitano,KirkLivingston} and have been successfully applied to several areas of link theory; we refer to~\cite{FriedlVidussi} for a survey.
In particular,  Morifuji used Fox calculus to prove that twisted Alexander polynomials satisfy a Torres condition~\cite{Morifuji}.
Reformulated in terms of twisted Reidemeister torsion, Morifuji's result states that, given a field~$\F$ and a representation~$\rho ' \colon \pi_1(S^3 \setminus L') \to \SL(n,\F)$,  there are elements~$\varepsilon_{k,\rho'} \in \F$ for~$k=1,\ldots,n-1$ such that
$$ \tau_L^\rho(t_1,\ldots,t_{\mu-1},1)\doteq\left( 
(t_1^{\ell_1}\cdots t_{\mu-1}^{\ell_{\mu-1}})^n 
+\sum_{k=1}^{n-1} \varepsilon_{k,\rho'} (t_1^{\ell_1}\cdots t_\mu^{\ell_{\mu-1}})^{n-k}
+(-1)^n\right)
\cdot\tau_{L'}^{\rho'}(t_1,\ldots,t_{\mu-1})~$$ 
where~$\ell_i:=\ell k(K_i,K_\mu)$ and~$\rho \colon \pi_1(S^3\setminus L)\to \pi_1(S^3 \setminus L') \to \SL(n,\F)$ is obtained by precomposing~$\rho'$ with the inclusion induced map.

The goal of this paper is to use twisted Reidemeister torsion to obtain a version of Morifuji's theorem 
%which allows for more general representations.
%The statement is also 
that is closer in appearance to Torres' original formula.
Continuing with the notation above and using \emph{UFD} as an abbreviation for ``Unique Factorization Domain"  our main result reads as follows.

\begin{theorem}
\label{thm:MainIntro}
Let~$R$ be a Noetherian UFD,  let~$\rho' \colon \pi_1(S^3 \setminus L') \to \GL(n,R)$ be a representation, let~$\rho \colon \pi_1(S^3\setminus L)\to \pi_1(S^3 \setminus L') \xrightarrow{\rho'} \GL(n,R)$ be the composition of~$\rho'$ with the inclusion induced map,   let~$[K_\mu] \in \pi_1(S^3 \setminus L')$ denote any based homotopy class of~$K_\mu$ in the complement of~$L'$.

The twisted Reidemeister torsion of~$(L,\rho)$ satisfies
$$ \tau_L^\rho(t_1,\ldots,t_{\mu-1},1)
= \det\left(t_1^{\ell_1}\cdots t_{\mu-1}^{\ell_{\mu-1}}\rho'([K_\mu])-I_n\right) \cdot \tau_{L'}^{\rho'}(t_1,\ldots,t_{\mu-1})
$$
up to multiplication by units of $R[t_1^{\pm 1},\ldots,t_{\mu-1}^{\pm 1}].$
\end{theorem}

When~$n=1,\rho'$ is the trivial representation and $R=\Z$, Theorem~\ref{thm:MainIntro} recovers the Torres formula for the (untwisted) Alexander polynomial whereas, as we explain in Corollary~\ref{mori},  in the case of~$\SL(n,\F)$ representations, it recovers Morifuji's formula, albeit with a less refined indeterminacy.

\begin{remark}
\label{rem:Indeterminacy}
Morifuji's twisted Torres formula for $\SL(n,\F)$-representations holds up to multiplication by $\pm t_1^{nk_1}\cdots t_\mu^{nk_\mu}$, whereas ours only holds up to multiplication by~$x \cdot t_1^{k_1}\cdots t_\mu^{k_\mu}$ with~$x \in \F^\times.$
%On a related note 
We believe that with a little more effort, the equality in Theorem~\ref{thm:MainIntro} could be shown to hold up to multiplication by elements of the form~$\pm \det(\rho'(g))t_1^{nk_1}\cdots t_\mu^{nk_\mu}$ with~$g\in \pi_1(S^3 \setminus L').$
%Also perhaps just domains because don't have to use orders.
\end{remark}

\begin{remark}
% We comment on the choice of based homotopy class~$[K_\mu] \in \pi_1(S^3 \setminus L')$ mentioned in Theorem~\ref{thm:MainIntro}. The based-free homotopy class of $K_\mu$ in $[S^1,\xlpr]$ is uniquely determined which corresponds to the conjugacy class of $\pi_1(\xlpr)$. Therefore, its image via $\rho'$ is a conjugacy class of $\GL(n,R)$ which uniquely determines its characteristic polynomial where the variable is $t_1^{\ell_1}\cdots t_{\mu-1}^{\ell_{\mu-1}}$.
The proof of Theorem~\ref{thm:MainIntro} shows that~$\det\left(t_1^{\ell_1}\cdots t_{\mu-1}^{\ell_{\mu-1}}\rho'([K_\mu])-I_n\right)$ does not depend on the based homotopy class of $K_\mu$ but this can also be verified directly by using the conjugacy invariance of the determinant. 
% We comment on the choice of based homotopy class~$[K_\mu] \in \pi_1(S^3 \setminus L')$ mentioned in Theorem~\ref{thm:MainIntro}.
% It is possible to choose any based homotopy class because different choices of paths from the basepoint of~$S^3 \setminus L'$ to~$K_\mu$ lead to conjugate elements~$g,g' \in \pi_1(S^3 \setminus L')$ and the conjugacy invariance of the determinant then implies that~$\det(T\rho(g)-I_n)=\det(T\rho(g')-I_n)$, where $T:=t_1^{\ell_1}\cdots t_{\mu-1}^{\ell_{\mu-1}}.$
% % This argument also shows that $\det(T\rho'([K_\mu])-I_n))$ only depends on the free homotopy class of $K_\mu$ in $X_{L'}$.
% % This can also be seen by noting that $\det(T\rho'([K_\mu])-I_n))$ is the characteristic polynomial of $\rho'([K_\mu])$.
\end{remark}

\subsection*{Organization}

Section~\ref{sec:Background} reviews some background material including the definition of Reidemeister torsion and Alexander polynomials.
Section~\ref{sec:Proof} focuses on the proof of Theorem~\ref{thm:MainIntro} and shows how it recovers Morifuji's theorem. 

\subsection*{Acknowledgments}

Thanks go to Anthony Conway for suggesting this project and for his comments on prior drafts of this article.

\subsection*{Conventions}

Unless mentioned otherwise, spaces are assumed to be path-connected and to have the homotopy type of a finite CW complex.
Manifolds are assumed to compact,  connected and oriented.
The unit group of a ring~$R$ is denoted by~$R^\times$ and the field of fractions of a domain $S$ is denoted by~$\widetilde{S}.$

\section{Background}

This section collects some background material needed to prove Theorem~\ref{thm:MainIntro}.
Section~\ref{sub:Twisted} reviews twisted (co)homology, Section~\ref{sub:Reidemeister} is concerned with the Reidemeister torsion of a chain complex, whereas Sections~\ref{sub:twistedAlexanderpolynomials} and~\ref{sub:TwistedReidemeister} focus on twisted Alexander polynomials and twisted Reidemeister torsion respectively.
 
\label{sec:Background}
\subsection{Twisted (co)homology}
\label{sub:Twisted}
This section reviews twisted homology and cohomology.
Our treatment follows~\cite{ConwayThesis}, but many more details on this material can be found in~\cite{FriedlLectureNotes}.
\medbreak

 Let~$X$ be a space and let~$Y\subset X$ be a possibly empty subspace. Denote by~$p:\widetilde{X}\map\ X$ the universal cover of~$X$ and set~$\widetilde{Y}:=p^{-1}(Y)$. The left action of~$\pi_1(X)$ by deck transformations on~$\widetilde{X}$ endows the singular chain complex~$C_*\pairtil$ with the structure of a left~$\zpi$-module. Let~$R$ be a ring and~$M$ be a~$(R,\zpi)$-bimodule. 
 We write~$\overline{C_*\pairtil}$ for the right~$\zpi$-module whose underlying abelian group agrees with~$C_*\pairtil$ but where the action is by~$\sigma \cdot g=\overline{g}\sigma$ for~$\sigma \in C_*\pairtil$ and~$g \in \pi_1(X)$.

\begin{definition}
\label{def:TwistedHomology}
    The chain complex~$C_*\pairm:=M\otimes_\zpi C_* \pairtil$ of left~$R$-modules is called the \emph{twisted chain complex} of~$\pair$ with coefficients in~$M$. The corresponding homology left~$R$-modules~$H_*\pairm$ are called the \emph{twisted homology modules} of~$\pair$ with coefficients in~$M$. The cochain complex~$C^*\pairm:=\Hom_{\text{right-}\zpi}(\overline{C_*\pairtil},M)$ of left~$R$-modules is called the \emph{twisted cochain complex} of~$\pair$ with coefficients in~$M$. The corresponding cohomology left~$R$-modules~$H^*\pairm$ are called the \emph{twisted cohomology modules} of~$\pair$ with coefficients in~$M$.
\end{definition}

If~$Y=\emptyset$, then we write~$H_*(X;M)$ and~$H^*(X;M)$ instead of~$H_*(X,\emptyset;M)$ and~$H^*(X,\emptyset;M)$. 
The isomorphism type of twisted (co)homology is independent of the choice of the basepoint in~$X$ for~$\fund$. 
\begin{remark}
\label{rem:CWTwisted}
If~$(X,A)$ is a CW pair, the same definition can be made using cellular (co)chain complexes and the resulting (co)chain complexes are chain homotopy equivalent~\cite[Proposition A.3]{friedl2024surveyfoundationsfourmanifoldtheory}. In the sequel, these identifications will be made implicitly.
\end{remark}

Given a representation~$\rho \colon\pi_1(X)\map \GL(n,R)$, we define a right action of~$\pi_1(X)$ on~$R^n$ by
\[
x\cdot g := \rho(g)^{-1}x
\] 
for each~$g\in \pi_1(X)$ and~$x\in R^n$. 
Here we think of elements of~$R^n$ as column vectors.
This action endows~$R^n$ with the structure of a right~$\Z[\pi_1(X)]$-module.
It follows that~$R^n$ is a~$(R,\Z[\pi_1(X)])$-bimodule which will be denoted by~$(R^n,\rho)$. 
For convenience, we write
\begin{align*}
    C_*(X,Y;R^n,\rho)&:=C_*(X,Y;(R^n,\rho)) &
    H_*(X,Y;R^n,\rho)&:=H_*(X,Y;(R^n,\rho)) \\
    C^*(X,Y;R^n,\rho)&:=C^*(X,Y;(R^n,\rho)) &
    H_*(X,Y;R^n,\rho)&:=H^*(X,Y;(R^n,\rho)).
\end{align*}
We refer to \cite[Lemma 2.7]{FrKim} for a proof of the following lemma.

\begin{lemma}\label{chi-twisted}
    If~$X$ is an~$m$-manifold and~$R$ is a domain, then \[\sum_{i=0}^m (-1)^i \rk_R(H_i(X;R^n,\rho))=n\cdot\chi(X).\]
\end{lemma}

We recall a description of the~$0^{th}$ twisted (co)homology module.

\begin{definition}\label{pi invariant}
    Let~$\pi$ be a group and let~$M$ be a~$(R,\Z[\pi])$-bimodule. Define the \emph{left~$R$-module of~$\pi$-coinvariants}  of~$M$ by 
    \[ M_\pi:=M\Bigg/ \left\lbrace\sum_{i=1}^{k}{(v_i\cdot g_i - v_i)} \ \Bigg| \ g_i\in\pi ,v_i\in M\right\rbrace \] 
    and the \emph{left~$R$-module of~$\pi$-invariants} of~$M$ by 
\[ M^\pi:= \left\lbrace v\in M \mid v\cdot g=v \text{ for all } g \in \pi \right\rbrace. \]
\end{definition}

\begin{proposition}\cite[Proposition 189.1]{FriedlLectureNotes}\label{H0}
    Let~$X$ be a path-connected space. We write~$\pi:=\pi_1(X)$. Let~$R$ be a ring and let~$M$ be a~$(R,\Z[\pi])$-bimodule. Then, we have the following isomorphisms of left~$R$-modules:
    \begin{align*}
        \displaystyle
        H_0(X;M) &\cong M_\pi \\
        \displaystyle
        H^0(X;M) &\cong M^\pi.
    \end{align*}
\end{proposition}

Finally, we recall the following fact from homological algebra. 
\begin{proposition}\label{tensor acyclicity-torsion}
    Let~$R$ be a domain and let~$C$ be a finite chain complex of~$R$-modules. Then, the following are equivalent:
    
    \begin{enumerate}[label=(\roman*)]
    \item~$\widetilde{R}\otimes_R C$ is acyclic.
    \item~$H_i(C)$ is~$R$-torsion for each~$i\geq 0$.
    \end{enumerate}
\end{proposition}

\subsection{Reidemeister torsion}
\label{sub:Reidemeister}
This section sets up our notation for Reidemeister torsion and recalls its relation to orders of modules.
We refer to~\cite{Tura} for a more lengthy discussion of the topic.

\medbreak

We refer to \cite[Chapter 1]{Tura} for the definition of the \textit{torsion}~$\tau(C) \in \F^\times$ of a finite acyclic chain complex~$C$ of finite dimensional based vector spaces over a field~$\F$. 
%In this case,~$\tau$ is well-defined in~$\F^\times$.
If~$C$ is not acyclic, then we set~$\tau(C)=0$. Hence, the torsion is an element in~$\F$. 
% Next, we recall two properties of the torsion.

% \begin{proposition}[{\cite[Theorem 1.5]{Tura}}]\label{multiplication law of torsions}
%     Let~$0\map C' \map C\map C''\map 0$ be a short exact sequence of finite acyclic chain complexes of finite dimensional based vector spaces. Then, \[\tau(C)=\pm\tau(C')\tau(C'').\]
% \end{proposition}

Given a Noetherian UFD~$R$ and an~$R$-module~$M$, we write~$\Delta(M)\in R/R^{\times}$ for the order of~$M$; we refer the reader to~\cite[Chapter 7]{Kawa} for the definition of~$\Delta(M).$
The next proposition recalls the relation between Reidemeister torsion and orders.

\begin{proposition}[{\cite[Theorem 4.7]{Tura}}]\label{torsion-order}
    Let~$R$ be a Noetherian UFD and let
    $$C_*=(C_m\dd{\map}C_0)$$
    be a based free chain complex of finite rank over~$R$ such that~$\rk(H_i(C_*))=0$ for all~$i\geq 0$. Then~$\widetilde{C}_*:=\widetilde{R}\otimes_R C_*$ is acyclic and satisfies the following formula:
    \[
    \displaystyle
    \tau(\widetilde{C}_*)=\displaystyle\prod_{i=0}^{m} \Delta(H_i (C_*))^{(-1)^{i+1}}.
    \]
\end{proposition}
Note that the equality in the above equation holds up to multiplication by units of~$R$.

\subsection{Twisted Alexander polynomials}
\label{sub:twistedAlexanderpolynomials}
This short section recalls the definition of twisted Alexander polynomials.
Helpful references on the topic include~\cite{FriedlVidussi} and~\cite{KirkLivingston}.
\medbreak

Let~$R$ be a Noetherian UFD. Let~$N$ be a 3-manifold with empty or toroidal boundary, let~$\psi \colon \pi_1(N)\map H$ be a non-trivial homomorphism to a free abelian group~$H$ and let~$\rho \colon \pi_1(N)\map \GL(n,R)$ be a representation. 
The tensor representation~$\rho\otimes\psi \colon \pi_1(N)\map \GL(n,R[H])$ is defined by~$g\mapsto \rho(g)\cdot\psi(g)$.
\begin{definition}
For~$i \geq 0$, the \emph{$i^{th}$ twisted Alexander polynomial}~$\Delta^{\rho\otimes\psi}_{N,i}$ of~$(N,\rho\otimes\psi)$ is 
%defined to be 
$$\Delta(H_i(N;R[H]^n,\rho\otimes\psi))\in R[H]/R[H]^\times$$
%for each~$i\geq 0$. 
\end{definition}

We now specialize to the case where~$N=X_L$ is the exterior of a~$\mu$-component link~$L \subset S^3$ (this is indeed a 3-manifold with toroidal boundary)
%Let~$L\subset S^3$ be a link with~$\mu$ components. Set~$X_L=S^3\setminus \nu L$, then~$X_L$ is a 3-manifold with toroidal boundary.
and~$\psi=\gamma \colon \pi_1(X_L)\twoheadrightarrow H_1(X_L)\cong\Z^\mu$ is the Hurewicz epimorphism. 
%Then,~$\gamma$ is non-trivial. Let~$R$ be a Noetherian UFD.
\begin{definition}
    The \emph{$i^{th}$ twisted Alexander polynomial}~$\Delta^\rho_{L,i}$ of a link~$L$ associated with the representation~$\rho \colon\pi_1(X_L)\map \GL(n,R)$ is defined as~$\Delta^{\rho\otimes\gamma}_{X_L,i}$ for each~$i\geq 0$. 
\end{definition}

\subsection{Twisted Reidemeister torsion}
\label{sub:TwistedReidemeister}

We recall the definition of the twisted Reidemeister torsion of a CW pair, referring to~\cite{FriedlVidussi} for further details.
In this section, we work with cellular chain complexes.
\medbreak

Let~$(X,Y)$ be a finite CW pair and let~$\varphi \colon \pi_1(X)\map \GL(n,S)$ be a representation where~$S$ is an integral domain. 
Then,~$S^n$ is a~$(S,\Z[\pi_1(X)])$-bimodule via the induced ring homomorphism~$\varphi:\Z[\pi_1(X)]\map M_n(S)$. The twisted chain complex~$S^n\otimes_{\Z[\pi_1(X)]} C_*(\widetilde{X},\widetilde{Y})$ consists of finitely generated free~$S$-modules.

Bases for these modules are obtained by lifting the cell structure of~$X$ to the universal cover, orienting and ordering the outcome, and combining the resulting basis of~$C_*(\widetilde{X},\widetilde{Y})$ with the canonical basis of~$S^n$. 

%oriented ordered cells~$\{\widetilde{e}^k_i\}$ of~$\widetilde{X}$ and the standard basis of~$S^n$. 
% By applying~$\widetilde{S}\otimes -$, we get a finite chain complex of finite dimensional based vector spaces.

% We have the canonical bases for the chains from the oriented ordered cells~$\{\widetilde{e}^k_i\}$ of~$\widetilde{X}$ and the standard basis of~$S^n$. 
% By applying~$\widetilde{S}\otimes -$, we get a finite chain complex of finite dimensional based vector spaces.

\begin{definition}
    The \emph{twisted Reidemeister torsion}~$\tau((X,Y),\varphi)\in \widetilde{S}$ of a CW pair~$(X,Y)$ associated with a representation~$\varphi:\pi_1(X)\map \GL(n,S)$ is defined as the torsion of the based chain complex~$\widetilde{S}^n\otimes C_*(\widetilde{X},\widetilde{Y})$.
\end{definition}

The next proposition resolves the ambiguity in the choice of the underlying CW-structure of the 3-manifold
%with empty or toroidal boundary 
and of the lifts of the cells.
%can be resolved via the following proposition.

\begin{proposition}[\cite{FriedlVidussi},\cite{Chap}] \label{indeterminacy of torsion}
    The twisted Reidemeister torsion~$\tau(N,\varphi)$ of a 3-manifold~$N$ with empty or toroidal boundary is well-defined up to multiplication by an element in \[
    \{\pm \det(\varphi(g))\,|\,g\in \pi_1(N)\}.
    \]
    In other words, up to this indeterminacy,~$\tau(N,\varphi)$ is independent of the choice of underlying CW-structure and the choice of the lifts of the cells.
\end{proposition}

In the sequel, given~$w\in \widetilde{S}$, we write
\[\tau(N,\varphi) \doteq_\tau w\]
if there exists a representative of~$\tau (N,\varphi)$ which equals~$w$. 

We recall the notation from Section~\ref{sub:twistedAlexanderpolynomials}.
Given a Noetherian UFD~$R$, the polynomial ring~$R[H]$ is also Noetherian UFD whose units are of the form~$r\cdot h$ where~$r\in R^\times$ and~$h\in H$. 
In the case where the representation~$\varphi$ is the tensor representation~$\rho\otimes\psi \colon \pi_1(N)\map \GL(n,R[H])$, by specializing Proposition~\ref{indeterminacy of torsion}, the indeterminacy of~$\tau(N,\rho\otimes\psi)$ is up to multiplication by an element in \[
\{\pm \det(\varphi(g))h\,|\,g\in \pi_1(N),h\in H\}.
\] From this point onwards, we regard the twisted
Reidemeister torsion of~$(N,\rho\otimes\psi)$ to be the value in~$\widetilde{R[H]}$ up to this indeterminacy. 

\begin{remark}
The twisted Alexander polynomial is closely related to the twisted Reidemeister torsion. For example, in the case that~$\rk(\text{Im}(\psi))\geq 2$, \cite[Proposition 3.2]{FriedlVidussi} shows that\[
\tau(N,\rho\otimes\psi) \doteq_\Delta\Delta^{\rho\otimes\psi}_{N,1},\]
where~$\doteq_\Delta$ indicates equality up to multiplication by units of~$R[H]$.
For this reason, the twisted Reidemeister torsion is occasionally referred to as the twisted Alexander polynomial; see e.g.~\cite{Wada}.
\end{remark}

Next, we consider the Reidemeister torsion of the exterior~$\xl$ of a link~$L \subset S^3$.
Recall that this is a 3-manifold with toroidal boundary and that~$\gamma \colon \pi_1(X_L)\twoheadrightarrow H_1(X_L)$ denotes the Hurewicz epimorphism.

\begin{definition}\label{definition of the torsion of link}
    The \emph{twisted Reidemeister torsion}~$\tau^\rho_L(t_1\dd{,}t_\mu)$ of a link~$L$ associated with a representation~$\rho \colon\pi_1(X_L)\map \GL(n,R)$ is defined as~$\tau(X_L,\rho\otimes\gamma)$.
\end{definition}

Since~$H_1(X_L)\cong \Z^\mu=\langle t_1\dd{,}t_\mu\rangle$ where~$t_i$ corresponds to the~$i^{th}$ meridian of~$L$,~$\tau(X_L,\rho\otimes\gamma)$ can be represented by a rational function of~$\mu$ variables~$t_1\dd{,}t_\mu$ over~$R$. Therefore, our notation for the twisted Reidemeister torsion is justified.

\section{Proof of the main theorem}
\label{sec:Proof}
\subsection{Preliminary lemmas}
\label{sub:PreliminaryLemmas}

This short section records two lemmas that will be needed in the proof of Theorem~\ref{thm:MainIntro}.
\medbreak

Let~$N$ be a 3-manifold with empty or toroidal boundary, let~$\psi:\pi_1(N)\map H$ be a non-trivial homomorphism to a free abelian group~$H$ and let~$\rho \colon\pi_1(N)\map \GL(n,R)$ be a representation. 
Furthermore let~$\phi \colon H\map J$ be a homomorphism to a free abelian group such that~$\phi\circ\psi$ is non-trivial.
We denote the induced ring homomorphism~$R[H]\map R[J]$ by~$\phi$ as well. 
Set 
\[S:=\{f\in R[H]\mid \phi(f)\neq 0\}=R[H]\setminus \ker\phi.\] 
Note that~$\phi$ induces a homomorphism~$S^{-1}R[H]\map \widetilde{R[J]}$ which we also denote by~$\phi$.
\begin{lemma}[\cite{FriedlVidussi} Proposition 3]\label{evalutation lemma}
     We have~$\tau(N,\rho\otimes\psi)\in S^{-1}R[H]$, and \[    \tau(N,\rho\otimes(\phi\circ\psi))\doteq_\tau\phi(\tau(N,\rho\otimes\psi)).
    \]
\end{lemma}

We fix some notation needed to state our second lemma. 
Let~$L=K_1\dd{\cup}K_{\mu}\subset S^3$ be a link with~$\mu\geq 2$ components.
Set~$L':=L\setminus K_{\mu}\subset S^3.$ 
Write the respective exteriors of~$L$ and~$L'$ as~$X_L:=S^3\setminus \nu L$ and~$X_{L'}:=S^3\setminus \nu L'$. Let~$\rho' \colon\pi_1(\xlpr)\map \GL(n,R)$ be a representation where~$R$ is a Noetherian UFD. Denote~$\Gamma:= R[H_1(\xlpr)]$. Consider the Hurewicz epimorphism~$\gamma':\pi_1(\xlpr)\twoheadrightarrow H_1(\xlpr)$ and the tensor representation~$\rho'\otimes\gamma':\pi_1(\xlpr)\map \GL(n,\Gamma)$. We define
\[
T:=\gamma'(K_\mu)=\prod_{i=1}^{\mu-1}t_i^{\ell k(K_i,K_\mu)}.
\]
The second technical lemma of this section now reads as follows. 
\begin{lemma}\label{homology of pair}
    The~$\Gamma^n$-homology of~$\xlpair$ is 
\[ H_i:=H_i(X_{L'},X_L;\Gamma^n,\rho'\otimes\gamma')\cong \begin{cases}
        (\Gamma^n)_{\pi_1(K_\mu)}=\Gamma^n/(T\rho'(K_{\mu})-I_n)\Gamma^n \quad &\text{for i=2,} \\
        (\Gamma^n)^{\pi_1(K_\mu)} \quad &\text{for i=3,} \\
        0 \quad &\text{otherwise.}
    \end{cases}
    \]
    % Here,~$\Gamma^n$ is a right~$\Z[\pi_1(K_\mu)]$-module along the ring homomorphism induced by the inclusion~$K_\mu\hookrightarrow\xlpr$.
        Here, the right~$\Z[\pi_1(X_{L'})]$-module $\Gamma^n$ is considered as a right~$\Z[\pi_1(K_\mu)]$-module using the inclusion induced ring homomorphism $\Z[\pi_1(K_\mu)] \to\Z[\pi_1(X_{L'})]$.

    Additionally, $\rk \ H_2=\rk \  H_3$ and
    \begin{itemize}
\item if~$\det(T\rho'(K_{\mu})-I_n)= 0$, then $\rk \ H_2=\rk \ H_3>0$,
\item if~$\det(T\rho'(K_{\mu})-I_n)\neq 0$, then all the~$H_i$ are~$\Gamma$-torsion and $H_3=0.$
    \end{itemize}
    % \begin{itemize}
    %     \item If~$\det(T\rho'(K_{\mu})-I_n)\neq 0$, then all the homology modules are~$\Gamma$-torsion and the third homology module is 0.
    %     \item If~$\det(T\rho'(K_{\mu})-I_n)= 0$, then the second and the third homology modules have the same positive rank.
    % \end{itemize}
\end{lemma}

\begin{proof}
Write~$H_i:=H_i(X_{L'},X_L;\Gamma^n,\rho'\otimes\gamma')$ for simplicity. 
Considering the cellular chain complex of the pair~$(X_{L'},X_L)$, we get~$H_i=0$ for~$i\geq 4$. 
Assume~$i\leq 3$ from this point onward.
Note that~$\Bar{\nu}K_\mu$ is a 3-manifold with boundary. We get the following~$\Gamma$-isomorphisms:
    \begin{align*}
        H_i(X_{L'},X_L;\Gamma^n,\rho'\otimes\gamma') &\cong H_i(\Bar{\nu}K_\mu,\partial\Bar{\nu}K_\mu;\Gamma^n,(\rho'\otimes\gamma')\circ \eta_*) &\text{by excision} \\
        &\cong H^{3-i}(\Bar{\nu}K_\mu;\Gamma^n,(\rho'\otimes\gamma')\circ \eta_*) &\text{by Poincaré duality}\\
        &\cong H^{3-i}(K_\mu;\Gamma^n,(\rho'\otimes\gamma')\circ \eta_*\circ \iota_*) &\text{by homotopy equivalence}\\
        &\cong H^{3-i}(K_\mu;\Gamma^n,(\rho'\otimes\gamma')\circ (\eta\circ\iota)_*) &\text{by functoriality}
    \end{align*}
where~$\iota \colon K_\mu \hookrightarrow \Bar{\nu}K_\mu$ and~$\eta \colon \Bar{\nu}K_\mu\hookrightarrow X_{L'}$ are the inclusions; these induce ring homomorphisms between the group rings 
%of~$\Z$ 
over the corresponding fundamental groups.
    
As~$K_\mu$ is a 1-manifold,~$H_i=0$ for~$i\leq 1$.
By Poincaré duality applied to~$K_\mu$ and Proposition~\ref{H0}, we have:
    \begin{align*}
        H_2 &\cong H_0(K_\mu;\Gamma^n,(\rho'\otimes\gamma')\circ (\eta\circ \iota)_*)\cong(\Gamma^n)_{\pi_1(K_\mu)} \\
        H_3 &\cong H^0(K_\mu;\Gamma^n,(\rho'\otimes\gamma')\circ (\eta\circ \iota)_*)\cong(\Gamma^n)^{\pi_1(K_\mu)}.
    \end{align*}
As~$\pi_1(K_\mu)$ is infinite cyclic, Definition \ref{pi invariant} implies that 
\[H_2 \cong (\Gamma^n)_{\pi_1(K_\mu)}\cong\Gamma^n/\langle(v\cdot\rho'\otimes\gamma'(K_\mu)^{-1}-v)\mid v\in \Gamma^n\rangle \cong \Gamma^n/(T\rho'(K_\mu)-I_n)\Gamma^n.\]
It remains to prove the last two assertions.
    
Since~$\xl$ and~$\xlpr$ are both~$3$-manifolds with toroidal boundary,~$\chi^\Gamma(X_L)=\chi(X_L)=0$ and~$\chi^\Gamma(\xlpr)=\chi(\xlpr)=0$ by Lemma \ref{chi-twisted}. The long exact sequence of the pair~$\xlpair$ now implies that~$\chi^\Gamma(X_{L'},X_L)=0$. Therefore, we have~$\rk  H_2=\rk \ H_3$.

    Denote~$A:=T\rho'(K_\mu)-I_n$
and observe that the exact sequence
$$ 0 \to \ker(A) \to R^n \xrightarrow{A} R^n \to H_2 \to 0$$
ensures that $\rk \ H_2=\rk \ker(A).$
 If~$\det A= 0$, then $\rk \ H_3=\rk \ H_2=\rk \ker(A)>0$.
 If~$\det A\neq 0$, then~$\ker(A)$=0 and therefore~$\rk \ H_3=\rk \ H_2=\rk \ker(A)=0$. 
 But since~$\Gamma^n$ is torsion-free, so is~$H_3\cong (\Gamma^n)^{\pi_1(K_\mu)}\subset \Gamma^n$.
 This implies~$H_3$ is both torsion and torsion-free, so~$H_3=0$.
    % If~$\det A\neq 0$, then~$A$ is injective, hence \(\Gamma^n\cong A \Gamma^n.\) Therefore,~$\rk \ H_2=\rk \ H_3=0$. On the other hand~$H_3\cong (\Gamma^n)^{\pi_1(K_\mu)}\subset \Gamma^n$ is torsion-free because~$\Gamma^n$ is torsion-free. This implies~$H_3$ is both torsion and torsion-free, so~$H_3=0$.
    
    % If~$\det A= 0$, then the nullspace has positive rank (i.e. the image has the rank smaller than~$n$). Hence,~$\rk \ H_2=\rk \ H_3>0.$
    \end{proof}

\subsection{Conclusion of the theorem}
\label{sub:ProofConclusion}

This section concludes the proof of our main result.
As in the previous sections, we set \[
    T:=\prod_{i=1}^{\mu-1}t_i^{\ell k(K_i,K_\mu)}.
    \]
and it is understood that all links are in~$S^3.$
We also recall that $\doteq_\Delta$ denotes equality up to multiplication by units of $R[H]=R[t_1^{\pm 1},\ldots,t_{\mu-1}^{\pm 1}].$

\begin{customthm}
{\ref{thm:MainIntro}}
Let~$L=K_1\dd{\cup}K_{\mu}$ be a link with~$\mu\geq 2$ components and let~$L':=L\setminus K_{\mu}$. 
Let~$\rho' \colon\pi_1(S^3\setminus \nu L')\map \GL(n,R)$ be a representation where~$R$ is a Noetherian UFD. Denote by~$\rho$ the representation~$\pi_1(S^3\setminus \nu L)\map\pi_1(S^3\setminus \nu L')\xrightarrow{\rho'} \GL(n,R)$.

The twisted Reidemeister torsion of~$(L,\rho)$ satisfies \[
    \displaystyle
    \tau^\rho_L(t_1\dd{,}t_{\mu-1},1)\doteq_\Delta \det(T\rho'(K_{\mu})-I_n)\cdot\tau^{\rho'}_{L'}(t_1\dd{,}t_{\mu-1}).\]
\end{customthm}

\begin{proof}
We begin by setting up some notation and by reformulating~$\tau^\rho_L(t_1\dd{,}t_{\mu-1},1)$ in a more convenient way.
    The inclusion~$\iota:X_L\hookrightarrow X_{L'}$ induces group epimorphisms on~$\pi_1$ and~$H_1$ that fit into the following commutative diagram: 
        \begin{equation}\label{commutative diagram}
            \begin{tikzcd}
            \pi_1(X_L) \arrow[r, two heads, "\iota_*"] \arrow[d, two heads, "\gamma"] & \pi_1(X_{L'}) \arrow[d, ,two heads, "\gamma'"] \\
            H_1(X_L) \arrow[r, two heads, "\iota_*"] & H_1(X_{L'}).
            \end{tikzcd}
            \begin{tikzcd}[row sep=-5pt]
            t_i \arrow[r, maps to] & t_i \\
            t_\mu \arrow[r, maps to] & 1    
            \end{tikzcd}
        \end{equation}
Here, for each~$i<\mu$, the variable~$t_i$ corresponds to the meridian of~$K_i$. 
By \eqref{commutative diagram} and because~$\rho \colon=\rho'\circ\iota_*$, we have: \begin{equation}\label{tensor-composition}
        (\rho'\otimes\gamma')\circ\iota_*=(\rho'\circ\iota_*)\otimes(\gamma'\circ\iota_*)=\rho\otimes(\iota_*\circ\gamma).
    \end{equation}
Note that~$H_1(\xl)$ and~$H_1(\xlpr)$ are free abelian groups and~$\iota_*\circ\gamma$ is non-trivial. 
We now reformulate~$\tau^\rho_L(t_1\dd{,}t_{\mu-1},1)$ as
        \begin{equation}\label{evalutation}
        \begin{split}
            \tau(X_L,(\rho'\otimes\gamma')\circ\iota_*) 
            &\doteq_\tau\tau(X_L,\rho\otimes(\iota_*\circ\gamma)) \quad\text{by \eqref{tensor-composition}} \\
            &\doteq_\tau \iota_*(\tau(X_L,\rho\otimes\gamma)) \quad\quad\text{by Lemma \ref{evalutation lemma} with }\psi=\gamma,\, \phi=\iota_*\\
            &\doteq_\tau\tau^\rho_L(t_1\dd{,}t_{\mu-1},1),
        \end{split}
        \end{equation}
    where in the last equality we used that the ring homomorphism \[\iota_*:(R[H_1(\xl)]\setminus \ker\iota_*)^{-1}R[H_1(\xl)]\map \widetilde{R[H_1(\xlpr)]}\]
    described in Lemma \ref{evalutation lemma} maps~$t_i$ to~$t_i$ for~$i<\mu$ but maps~$t_\mu$ to~$1$ as in \eqref{commutative diagram}.

One last piece of notation: in what follows, we use the following abbreviations for simplicity:
        \begin{align*}
            C_*(\xlpr)&:=C_*(\xlpr;\Gamma^n,\rho'\otimes\gamma') &
            H_*(\xlpr)&:=H_*(\xlpr;\Gamma^n,\rho'\otimes\gamma') \\
            C_*(\xlpr,\xl)&:=C_*(\xlpr,\xl;\Gamma^n,\rho'\otimes\gamma') &
            H_*(\xlpr,\xl)&:=H_*(\xlpr,\xl;\Gamma^n,\rho'\otimes\gamma') \\
            C_*(\xl)&:=C_*(\xl;\Gamma^n,(\rho'\otimes\gamma')\circ \iota_*) &
            H_*(\xl)&:=H_*(\xl;\Gamma^n,(\rho'\otimes\gamma')\circ \iota_*).
        \end{align*}
We are now ready to prove the twisted Torres formula.
    \begin{case}[$\det(T\rho'(K_{\mu})-I_n)= 0$]
        It is enough to show that~$\tau^\rho_L(t_1\dd{,}t_{\mu-1},1)=0$.
        By \eqref{evalutation}, we have left to show~$\GamTen C_*(X_L)$ is not acyclic which is equivalent to saying that~$H_i(X_L)$ is not~$\Gamma$-torsion for some~$i$ by Proposition \ref{tensor acyclicity-torsion}.
        
        As~$\xlpr$ has the homotopy type of a finite 2-dimensional CW complex,~$H_3(X_{L'})=0$ (see e.g.~\cite[Theorem 3.16]{friedl2024surveyfoundationsfourmanifoldtheory}). Using Lemma \ref{homology of pair},~$H_3(X_{L'},X_L)$ has positive rank and therefore the following part of the long exact sequence of the pair~$\xlpair$ shows that~$H_2(X_L)$ has a positive rank:\[
        \overbrace{H_3(\xlpr)}^{=0}\map H_3(X_{L'},X_L)\map H_2(X_L).
        \]
Thus, in this case, the Torres formula holds trivially: both sides of the equality vanish.
    \end{case}
    \begin{case}[$\det(T\rho'(K_{\mu})-I_n)\neq 0$]
 We assert that~$\GamTen C_*(X_L)$ is acyclic if and only if~$\GamTen C_*(\xlpr)$ is acyclic. 
Recall that~$H_i(\xl,
        \xlpr)$ is~$\Gamma$-torsion for all~$i$ from Lemma \ref{homology of pair}, equivalently, we have~$\GamTen H_i(\xl,\xlpr)=0$ for all~$i$.
By applying~$\GamTen -$ to the long exact sequence of homology~$\Gamma$-modules of the pair~$\xlpair$, we get that~$H_i(\xl)$ and~$H_i(\xlpr)$ have the same rank for all~$i$.
        By Proposition~\ref{tensor acyclicity-torsion}, the assertion holds.
%%%
\begin{subcase}[$\tau^{\rho'}_{L'}(t_1\dd{,}t_{\mu-1})=0$]
This condition is equivalent to~$\GamTen C_*(\xlpr)$ not being acyclic. Hence,~$\GamTen C_*(\xl)$ is not acyclic as well by the assertion which is equivalent to~$\tau^\rho_L(t_1\dd{,}t_{\mu-1},1)\doteq_\tau \tau(X_L,(\rho'\otimes\gamma')\circ\iota_*)=0$ where the first equality is \eqref{evalutation}.
Thus, in this case, the Torres formula holds trivially once again.
\end{subcase}
%%%%%
\begin{subcase}[$\tau^{\rho'}_{L'}(t_1\dd{,}t_{\mu-1})\neq 0$]
By the assertion, both~$\GamTen C_*(\xlpr)$ and~$\GamTen C_*(\xl)$ are acyclic. 
Choose cell structures on $X_{L'}$ and $X_L$ so that $(X_{L'},X_L)$ is a CW pair.
%, we get a relative CW structure on the pair $(X_L',X_L)$.
After picking the canonical basis for~$\widetilde{\Gamma}^n$, lifts, orientations and an ordering for the cells of~$X_{L'}$, we get bases~$\mathbf{c}_{X_{L'}},\mathbf{c}_{X_L}$ and~$\mathbf{c}_{(X_{L'},X_L)}$ for the twisted chain complexes~$\widetilde{\Gamma} \otimes_\Gamma C_*(X_{L'}),\widetilde{\Gamma} \otimes_\Gamma C_*(X_L)$ and~$\widetilde{\Gamma} \otimes_\Gamma C_*(X_{L'},X_L)$.
In the notation of~\cite[page 4]{Tura}, these bases satisfy~$\mathbf{c}_{X_{L'}} \sim \mathbf{c}_{X_L}\mathbf{c}_{(X_{L'},X_L)}$ and so the multiplicativity of Reidemeister torsion~\cite[Proposition 1.5]{Tura} ensures that
%By Proposition \ref{multiplication law of torsions}, we get that: 
%Writing (X_L',X_L) means that we picked CW structures so that the pair is a CW pair.
%We get a relative CW structure for (X_L',X_L)
%After picking the canonical basis for \widetilde{\Gamma}^n, and lifts and orientations for the central  one,  we get bases for the twisted chain complexes.
%These satisfy c\simeq c'c''
%So can apply Turaev and then to remove choice of the bases, write =_\tau.
\begin{equation}\label{torres equation}
\tau(\xlpr,\rho'\otimes\gamma')\doteq_\tau\tau(\xl,(\rho'\otimes\gamma')\circ \iota_*)\cdot\tau(\xlpair,\rho'\otimes\gamma').  \end{equation}
% In Proposition \ref{multiplication law of torsions}, the equality was up to sign but this sign is already contained in the indeterminacy of the twisted Reidemeister torsion.
We compute the torsion of the pair~$\xlpair$ as follows:
\begin{align*}
\tau(\xlpair,\rho'\otimes\gamma')
&\doteq_\tau
\tau(\GamTen C_*\xlpair)\\ &\doteq_\Delta\displaystyle\prod_{i=0}^{m} \Delta(H_i (C_*\xlpair))^{(-1)^{i+1}} &\text{by Proposition \ref{torsion-order}}\\ &\doteq_\Delta \frac{1}{\Delta(H_2(\xlpr,\xl))} &\text{by Lemma \ref{homology of pair}}\\
&\doteq_\Delta \frac{1}  {\det(T\rho'(K_{\mu})-I_n)}.
\end{align*}
The last equality uses that the square matrix~$T\rho'(K_{\mu})-I_n$ presents~$H_2(\xlpr,\xl)$.
            
We know that 
\begin{align*}
\tau(X_L,(\rho'\otimes\gamma')\circ\iota_*)&\doteq_\tau \tau^\rho_L(t_1\dd{,}t_{\mu-1},1)&\text{by } \eqref{evalutation}\\
\tau(\xlpr,\rho'\otimes\gamma')&\doteq_\tau \tau^{\rho'}_{L'}(t_1\dd{,}t_{\mu-1}) &\text{by Definition \ref{definition of the torsion of link}.}
\end{align*} 
Therefore, equation \eqref{torres equation} can be reformulated to 
\[\tau^{\rho'}_{L'}(t_1\dd{,}t_{\mu-1})\doteq_\Delta \tau^\rho_L(t_1\dd{,}t_{\mu-1},1) \cdot\frac{1}  {\det(T\rho'(K_{\mu})-I_n)}.\]
We get the desired result by multiplying by~$\det(T\rho'(K_{\mu})-I_n)$ on both sides.
\end{subcase}
\end{case}
\end{proof}

\begin{remark}\label{rmk}
Since~$\det(T\rho'(K_{\mu})-I_n)$ is the characteristic polynomial of the invertible matrix~$\rho'(K_{\mu})$, it follows  that~$\det(T\rho'(K_{\mu})-I_n)$ is an element of~$R[T]$ (the polynomial ring over~$R$ in the variable~$T$) of degree~$n$ whose constant term is~$(-1)^n$.
 %    We argue that~$\det(T\rho'(K_{\mu})-I_n)$ is an element of~$R[T]$ (the polynomial ring over~$R$ in the variable~$T$) of degree~$n$ whose constant term is~$(-1)^n$.

 % Each entry of the matrix~$T\rho'(K_{\mu})-I_n$ is a linear polynomial in~$R[T]$. 
 % Therefore, the determinant is a polynomial in~$R[T]$ and its degree is less than or equal to~$n$. 
 % Since~$\det\rho'(K_{\mu})\neq 0$,
 % %%because GL
 % the degree~$n$ coefficient is nonzero. 
 % Hence,~$\det(T\rho'(K_{\mu})-I_n)$ is a polynomial of degree~$n$. 
 % The constant term is~$\det(-I_n)=(-1)^n$.
 \end{remark}

We recover Morifuji's result (albeit with a less refined indeterminacy) as follows.
\begin{corollary}[\cite{Morifuji} Theorem 3.6]\label{mori}
    Let~$L$ and~$L'$ be two links as in Theorem \ref{thm:MainIntro}, let $\F$ be a field, let~$\rho'\colon \pi_1(S^3\setminus \nu L')\map  \SL(n,\F)$ be a representation, and consider the composition~$\rho \colon \pi_1(S^3\setminus \nu L)\map\pi_1(S^3\setminus \nu L')\xrightarrow{\rho'}  \SL(n,\F)$.
    
The twisted Reidemeister torsion of~$(L,\rho)$ satisfies \[
    \displaystyle
    \tau^\rho_L(t_1\dd{,}t_{\mu-1},1)\doteq_\Delta \left( T^n+\sum_{i=1}^{\mu-1}\epsilon_i T^i+(-1)^n\right)\cdot\tau^{\rho'}_{L'}(t_1\dd{,}t_{\mu-1})
    \] 
    % where\[
    % T:=\prod_{i=1}^{\mu-1}t_i^{\ell k(K_i,K_\mu)}
    % \]and
where~$\epsilon_i\in \F$ for each~$1\leq i\leq n-1$. 
\end{corollary}

\begin{proof}
    Since every field is a Noetherian UFD, we can apply Theorem \ref{thm:MainIntro}. Since~$\rho'\in  \SL(n,\F)$, the polynomial~$\det(T\rho'(K_{\mu})-I_n)$ is monic. We've shown that the constant term is~$(-1)^n$ in Remark~\ref{rmk}.
\end{proof}

% \newpage

% Eigenvalue remark

% Torsion of pair and or solid torus.

\bibliographystyle{alpha}
\bibliography{BiblioTorres}
\end{document}